\documentclass[11pt]{extarticle}

\usepackage[T1]{fontenc}
\usepackage[british]{babel}
\usepackage{amsmath,amssymb,amsthm}
\usepackage{dsfont}
\usepackage{microtype}
\usepackage[top=1in, bottom=1.25in, left=0.75in, right=0.75in]{geometry}
\usepackage{authblk}
\usepackage{hyperref}

\hypersetup{
  colorlinks=true,
  linkcolor=blue,
  citecolor=blue,
  urlcolor=blue
}

\newtheorem{theorem}{Theorem}[section]
\newtheorem{lemma}[theorem]{Lemma}
\newtheorem{corollary}[theorem]{Corollary}

\theoremstyle{definition}

\newtheorem{proposition}[theorem]{Proposition}

\theoremstyle{remark}

\numberwithin{equation}{section}

\newcommand{\C}{\mathbb{C}}
\newcommand{\R}{\mathbb{R}}
\newcommand{\N}{\mathbb{N}}
\newcommand{\Sn}{\mathfrak{S}}
\newcommand{\one}{\mathbf 1}

\newcommand{\triv}{\mathrm{triv}}
\newcommand{\ip}[2]{\left\langle #1,\,#2\right\rangle}
\newcommand{\norm}[1]{\left\lVert #1\right\rVert}
\newcommand{\abs}[1]{\left\lvert #1\right\rvert}
\newcommand{\set}[1]{\left\{#1\right\}}
\newcommand{\Std}{\operatorname{Std}}
\newcommand{\Reg}{\operatorname{Reg}}
\newcommand{\Irr}{\operatorname{Irr}}

\newcommand{\Sph}{\operatorname{Sph}}

\newcommand{\len}{\ell}

\newcommand{\eps}{\varepsilon}
\newcommand{\Schur}{\mathbb S}

\DeclareMathOperator{\End}{End}
\DeclareMathOperator{\Ind}{Ind}
\DeclareMathOperator{\Symm}{Sym}

\begin{document}

\title{Random walks on wreath products \\and spectral gaps for coloured interchange processes}

\author{ Haoran Zhu\thanks{Division of Mathematical Sciences, School of Physical and Mathematical Sciences, Nanyang Technological University, 21 Nanyang Link, Singapore 637371. Email: zhuh0031@e.ntu.edu.sg}}

\date{}

\maketitle

\begin{abstract}
We introduce group-valued coloured interchange processes, a class of continuous-time random walks on wreath products generated by transpositions and arbitrary symmetric base-group transitions.  We give a complete representation-theoretic characterisation of their spectral gaps by decomposing the associated quasi-regular representation and determining precisely which irreducible representations are required.  For abelian base groups, we realise these representations on multislices and identify their Laplacians with discrete Schr\"odinger operators.  We then classify the minimal families of irreducible representations that determine the spectral gap for every choice of transition rates.
\end{abstract}

\medskip


\section{Introduction}
Let $G$ be a finite group, let $n\geqslant2$, and write $W_n(G)=G\wr\Sn_n=G^n\rtimes\Sn_n$.  We introduce a class of continuous-time random walks on $W_n(G)$.  From a state $u\in W_n(G)$, the walk jumps to $(ij)u$ at rate $\mathsf a_{ij}$ and to $gu$ at rate $\mathsf c_g$.  The rates are encoded by
\begin{equation}\label{eq:intro-weight}
w=w_T+w_N,
\qquad
w_T=\sum_{1\leqslant i<j\leqslant n}\mathsf a_{ij}(ij),
\qquad
w_N=\sum_{g\in G^n}\mathsf c_g\,g,
\end{equation}
where $\mathsf a_{ij},\mathsf c_g\geqslant0$ and $\mathsf c_g=\mathsf c_{g^{-1}}$.  We call such a walk a group-valued \textbf{coloured interchange process}.  The term coloured interchange process has also been used for a different recolouring process~\cite[Section~2.3]{TranThesis}; in the present setting the internal states are group-valued, and one transition may act on several coordinates.  We determine its spectral gap for every choice of the rates in \eqref{eq:intro-weight}.

Our starting point is \textit{Aldous' spectral gap conjecture}.  For a weighted graph on $n$ vertices, the interchange process is the continuous-time random walk on $\Sn_n$ that exchanges the labels at the ends of each edge at the prescribed rate.  Aldous~\cite{Aldous} conjectured that its spectral gap equals that of the corresponding one-particle random walk, and Caputo, Liggett, and Richthammer proved this for every weighted graph~\cite[Theorem~1.1]{CaputoLiggettRichthammer}.

Equivalently, for every non-negative linear combination of transpositions, the least non-zero eigenvalue in the regular representation of $\Sn_n$ is attained by the standard representation.  This formulation has motivated extensions to other elements of $\R[\Sn_n]$ and to other groups; see~\cite{AlonKozma,Cesi2016,ParzanchevskiPuder,AlonKozmaPuder}.

Wreath products provide a natural setting for interchange processes with internal states.  Random walks and mixing on wreath products were studied by Schoolfield~\cite{Schoolfield} and by Fill and Schoolfield~\cite[Sections~1.1--1.2]{FillSchoolfield}; the latter gave a probabilistic reduction for broad classes of Markov chains on $G\wr\Sn_n$ and related homogeneous spaces.

Recently, Cesi proved an Aldous-type theorem for $C_2\wr\Sn_n$ when the base-group part is supported on elements that are non-trivial in only one coordinate~\cite[Theorem~1.2]{Cesi2020}.  Ghosh extended this result to $G\wr\Sn_n$ for arbitrary finite $G$~\cite[Theorem~1.1]{Ghosh}.  Alon and Ghosh subsequently allowed arbitrary symmetric elements of $C_2^n$ and characterised the irreducible representations required to determine the gap~\cite[Theorems~1 and~3]{AlonGhosh}.  Further Aldous-type results for generalised symmetric groups and unitary groups appear in~\cite{LevhariPuder,AlonPuder}.

We allow the base-group term to be supported on arbitrary elements of $G^n$, so one transition may act on several coordinates.  This includes the single-coordinate processes above and, when $G=C_2$, the unrestricted signed interchange process of Alon and Ghosh~\cite[Theorem~1]{AlonGhosh}.  For abelian $G$, we diagonalise the base-group Laplacian by Fourier transform; for general $G$, we use the spherical Hecke algebra.

\subsection{Main results}

We now state the results precisely.  Let $\Std_n$ denote the standard representation of $\Sn_n$, regarded as a representation of $W_n(G)$ through the quotient $W_n(G)\to\Sn_n$.  Write $\mathcal P_n(G)=\Ind_{\Sn_n}^{W_n(G)}\triv$ for the quasi-regular representation on $W_n(G)/\Sn_n$, and let $\mathcal P_n^0(G)$ be the orthogonal complement of its one-dimensional fixed subspace.

\begin{theorem}\label{thm:intro-reduction}
Let $G$ be a finite group, let $n\geqslant2$, and let $w$ be as in \eqref{eq:intro-weight}.  Then
\[
\psi_{W_n(G)}(w)
=
\psi_{W_n(G)}\bigl(w,\Std_n\oplus \mathcal P_n^0(G)\bigr).
\]
\end{theorem}

The representation in Theorem~\ref{thm:intro-reduction} has dimension $\abs{G}^n+n-2$, whereas the regular representation has dimension $\abs{G}^n n!$.

We next decompose $\mathcal P_n(G)$.  Put $\Gamma=\Irr(G)$, and let $d_\gamma$ be the degree of $\gamma\in\Gamma$.  A $G$-multipartition of $n$ is a family $\Lambda=(\lambda^\gamma)_{\gamma\in\Gamma}$ of partitions with total size $n$; write $V_\Lambda$ for the corresponding irreducible representation of $W_n(G)$.  We use $\Schur_\lambda(\C^d)$ for the irreducible polynomial $\mathrm{GL}_d(\C)$-module of highest weight $\lambda$, with value zero when $\len(\lambda)>d$; see~\cite[Section~6.1]{FultonHarris}.

\begin{theorem}
Let $\Lambda=(\lambda^\gamma)_{\gamma\in\Gamma}$ be a $G$-multipartition of $n$.  Then
\[
\dim V_\Lambda^{\Sn_n}
=
\prod_{\gamma\in\Gamma}
\dim \Schur_{\lambda^\gamma}(\C^{d_\gamma}).
\]
Moreover, $V_\Lambda$ occurs in $\mathcal P_n(G)$ if and only if $\len(\lambda^\gamma)\leqslant d_\gamma$ for every $\gamma\in\Gamma$.
\end{theorem}

When $G$ is abelian, write $\widehat G$ for its character group.  The spherical multipartitions have one row in each component and are indexed by multiplicity vectors $\kappa=(\kappa_\chi)_{\chi\in\widehat G}$ satisfying $\sum_\chi\kappa_\chi=n$.  We write $V_\kappa$ for the corresponding irreducible representation.

\begin{theorem}
Let $G$ be a finite abelian group, let $n\geqslant2$, and let $w$ be as in \eqref{eq:intro-weight}.  Then
\[
\psi_{G\wr\Sn_n}(w)
=
\min\left\{
\psi(w,\Std_n),
\min_{\substack{\sum_{\chi\in\widehat G}\kappa_\chi=n\\
V_\kappa\neq\triv}}
\psi(w,V_\kappa)
\right\}.
\]
\end{theorem}

When $G$ is abelian, the quasi-regular representation is multiplicity-free.  We realise each of its irreducible summands on a multislice, a space that also occurs in the study of transposition chains and logarithmic Sobolev inequalities~\cite[Section~1]{FilmusODonnellWu} and~\cite[Section~1.1]{Salez}.  The resulting Laplacian is a discrete Schr\"odinger operator, namely a graph Laplacian plus a potential~\cite[p.~141]{SySunada}.  Fourier inversion then gives strictly positive weights that separate any prescribed duality orbit.

Our final main result establishes sharpness.

\begin{theorem}\label{thm:intro-minimality}
Let $G$ be a finite group and let $V$ be a non-trivial $\Sn_n$-spherical irreducible representation of $W_n(G)$.  Then there is an element $w$ of the form \eqref{eq:intro-weight} for which $\psi_{W_n(G)}(w,U)=\psi_{W_n(G)}(w)$ holds precisely when $U\cong V$ or $U\cong V^*$.
\end{theorem}

The base-group weights in Theorem~\ref{thm:intro-minimality} may be chosen strictly positive, and all transpositions may be given the same positive weight.  If $V\cong V^*$, then no other non-trivial irreducible representation has the same gap; the same conclusion holds for $\Std_n$.

Consequently, a minimal gap-determining family consists of $\Std_n$ together with one representative from each duality orbit of non-trivial spherical irreducible representations.  For an abelian group, duality sends $(\kappa_\chi)_\chi$ to $(\kappa_{\chi^{-1}})_\chi$.  When $G=C_2$, every multiplicity vector is fixed, and we recover the unrestricted Alon--Ghosh characterisation~\cite[Theorem~3]{AlonGhosh}.

In Section~\ref{sec:preliminaries}, we fix the notation for representation Laplacians and wreath products.  In Section~\ref{sec:reduction}, we reduce the spectral-gap problem to the lifted standard representation and the spherical irreducible representations.  In Section~\ref{sec:constituents}, we decompose the quasi-regular representation.  In Section~\ref{sec:cyclic-blocks}, we give the multislice realisation for abelian base groups.  In Section~\ref{sec:minimality}, we prove sharpness and classify the minimal gap-determining families.

\section{Weighted Laplacians and wreath products}\label{sec:preliminaries}

This section fixes the notation for finite-group Laplacians, wreath products, and the quasi-regular representation used throughout.

\subsection{Laplacians in finite group algebras}

All representations are finite-dimensional over $\C$ and are taken to be unitary; see~\cite[Section~4.3]{Serre}.  Our conventions for representation Laplacians and spectral gaps follow~\cite[Sections~1--2]{Cesi2016}.

Let $H$ be a finite group.  We call $w=\sum_{h\in H}\mathsf c_h\,h\in\R[H]$ symmetric and non-negative when $\mathsf c_h\geqslant0$ and $\mathsf c_h=\mathsf c_{h^{-1}}$ for every $h\in H$, and define $\Delta_H(w)=\sum_{h\in H}\mathsf c_h(1_H-h)$.
The coefficient of $1_H$ has no effect on $\Delta_H(w)$, but allowing it is convenient when constructing weights by Fourier inversion.

For a unitary representation $(\rho,V)$, symmetry of $w$ gives
\begin{equation}\label{eq:dirichlet-form}
\ip{\rho(\Delta_H(w))v}{v}
=
\frac12\sum_{h\in H}\mathsf c_h\norm{v-\rho(h)v}^2
\qquad (v\in V).
\end{equation}
In particular, $\rho(\Delta_H(w))$ is self-adjoint and positive semidefinite.

Let $V^H$ be the $H$-fixed subspace of $V$.  The spectral gap of $w$ in $V$, denoted by $\psi_H(w,V)$, is the smallest eigenvalue of $\rho(\Delta_H(w))$ on $(V^H)^\perp$.  If $V=V^H$, we set $\psi_H(w,V)=+\infty$.  We write $\psi_H(w)=\psi_H(w,\Reg_H)$ for the spectral gap of $w$.

Complete reducibility~\cite[Section~1.3]{Serre} gives
\begin{equation}\label{eq:gap-irrep-minimum}
\psi_H(w)=\min_{V\in\Irr(H)\setminus\{\triv\}}\psi_H(w,V).
\end{equation}
If $V$ is non-trivial and irreducible, then $V^H=0$, so $\psi_H(w,V)$ is the least eigenvalue of $V(\Delta_H(w))$.

\subsection{The wreath product and its quasi-regular representation}\label{subsec:wreath-quasiregular}

Let $G$ be a finite group with identity element $e$, and let $n\geqslant2$.  We write $W_n(G)=G^n\rtimes\Sn_n$ for the wreath product of $G$ by $\Sn_n$; see~\cite[Section~2.1]{CeccheriniScarabottiTolli}.  For $\pi\in\Sn_n$ and $g=(g_1,\ldots,g_n)\in G^n$, set $\pi\cdot g=(g_{\pi^{-1}(1)},\ldots,g_{\pi^{-1}(n)})$.  Multiplication is given by $(g;\pi)(h;\sigma)=\bigl(g(\pi\cdot h);\pi\sigma\bigr)$.
We use the canonical embeddings of $G^n$ and $\Sn_n$ in $W_n(G)$, and write $(ij)$ for the element corresponding to the transposition $(ij)\in\Sn_n$.

Let $\Std_n$ be the irreducible $\Sn_n$-representation of shape $(n-1,1)$.  We use the same notation for its lift through $W_n(G)\to\Sn_n$; thus $G^n$ acts trivially on $\Std_n$.

The quasi-regular representation of $W_n(G)$ on $W_n(G)/\Sn_n$ is $\mathcal P_n(G)=\Ind_{\Sn_n}^{W_n(G)}\triv\cong\C[W_n(G)/\Sn_n]$; see~\cite[Section~1.2.1]{CeccheriniScarabottiTolli}.  Every coset has a unique representative in $G^n$, so $(g;\pi)\Sn_n\mapsto\delta_g$ identifies $\mathcal P_n(G)$ with $\C[G^n]$.  Under this identification, $h\cdot\delta_g=\delta_{hg}$ for $h\in G^n$, while $\sigma\cdot\delta_g=\delta_{\sigma\cdot g}$ for $\sigma\in\Sn_n$.
The constant vectors form the unique fixed line.  We denote its orthogonal complement by $\mathcal P_n^0(G)$; in particular, $\dim\mathcal P_n(G)=\abs{G}^n$.

We shall use the following representation-theoretic form of Aldous' theorem.

\begin{theorem}[{\cite[Theorem~1.1]{CaputoLiggettRichthammer}}]\label{thm:aldous-form}
Let $u=\sum_{1\leqslant i<j\leqslant n}\mathsf a_{ij}(ij)\in\R[\Sn_n]$, where $\mathsf a_{ij}\geqslant0$, and let $U$ be a non-trivial irreducible representation of $\Sn_n$.  Then $\psi_{\Sn_n}(u,U)\geqslant\psi_{\Sn_n}(u,\Std_n)$.
\end{theorem}

\section{Reduction to spherical representations}\label{sec:reduction}

In this section we reduce the spectral-gap problem to the lifted standard representation and the spherical irreducible representations.  We use the standard terminology for spherical representations and finite homogeneous spaces from~\cite[Sections~1.2.1--1.2.2]{CeccheriniScarabottiTolli}.

An irreducible representation $V$ of $W_n(G)$ is $\Sn_n$-spherical when $V^{\Sn_n}\neq0$, and $\Sph_n(G)$ denotes the set of such irreducible representations.

\begin{lemma}[Frobenius reciprocity; {\cite[Section~1.2.1]{CeccheriniScarabottiTolli}}]\label{lem:frobenius-spherical}
Let $V$ be an irreducible representation of $W_n(G)$.  Then $V$ is $\Sn_n$-spherical if and only if it occurs in $\mathcal P_n(G)$, and its multiplicity is $\dim V^{\Sn_n}$.
\end{lemma}

\begin{proposition}\label{prop:domination}
Let $w$ be as in \eqref{eq:intro-weight}, and let $V$ be a non-trivial irreducible representation of $W_n(G)$ with $V^{\Sn_n}=0$.  Then $\psi_{W_n(G)}(w,V)\geqslant\psi_{W_n(G)}(w,\Std_n)$.
\end{proposition}

\begin{proof}
Since $V^{\Sn_n}=0$, every irreducible constituent of $V\!\downarrow_{\Sn_n}$ is non-trivial.  Theorem~\ref{thm:aldous-form} therefore bounds the transposition Laplacian below by $\psi_{\Sn_n}(w_T,\Std_n)I$.  The base-group Laplacian is positive semidefinite by \eqref{eq:dirichlet-form}, and hence
\[
\psi_{W_n(G)}(w,V)
\geqslant
\lambda_{\min}\bigl(V(\Delta_{W_n(G)}(w_T))\bigr)
\geqslant
\psi_{\Sn_n}(w_T,\Std_n).
\]
The base group $G^n$ acts trivially on $\Std_n$, so $\Std_n(\Delta(w_N))=0$ and $\psi_{\Sn_n}(w_T,\Std_n)=\psi_{W_n(G)}(w,\Std_n)$.
\end{proof}

Combining \eqref{eq:gap-irrep-minimum}, Lemma~\ref{lem:frobenius-spherical}, and Proposition~\ref{prop:domination}, we obtain the spectral reduction.

\begin{theorem}\label{thm:universal-reduction}
Let $G$ be a finite group, let $n\geqslant2$, and let $w$ be as in \eqref{eq:intro-weight}.  Then
\[
\psi_{W_n(G)}(w)
=
\min\left\{
\psi_{W_n(G)}(w,\Std_n),
\min_{\substack{V\in\Sph_n(G)\\V\neq\triv}}
\psi_{W_n(G)}(w,V)
\right\}.
\]
\end{theorem}

Since the non-trivial spherical representations are the irreducible summands of $\mathcal P_n^0(G)$, Theorem~\ref{thm:universal-reduction} may be written equivalently as
\[
\psi_{W_n(G)}(w)
=
\psi_{W_n(G)}\bigl(w,\Std_n\oplus \mathcal P_n^0(G)\bigr).
\]

Since $\dim \mathcal P_n^0(G)=\abs{G}^n-1$ and $\dim\Std_n=n-1$, we also obtain the following dimension bound.

\begin{corollary}
Let $G$ be a finite group and let $n\geqslant2$.  Then the spectral gap of every $w$ in \eqref{eq:intro-weight} may be computed in a representation of dimension $\abs{G}^n+n-2$.
\end{corollary}

The reduction also has a direct Markov-chain interpretation.  Let $L_{\mathsf a}$ be the weighted graph Laplacian on $\C^n$ defined by $(L_{\mathsf a}f)(i)=\sum_{j\neq i}\mathsf a_{ij}(f(i)-f(j))$, and write $\lambda_{\mathrm{rw}}(\mathsf a)=\lambda_{\min}(L_{\mathsf a}\vert_{\one^\perp})$.
If the weighted graph with edge weights $\mathsf a_{ij}$ is connected, then $\lambda_{\mathrm{rw}}(\mathsf a)$ is its usual spectral gap.

On $\C[G^n]$, regarded as the space of functions on $G^n$, define
\begin{equation}\label{eq:colour-configuration-laplacian}
(\mathcal L_{\mathrm{col}}F)(x)
=
\sum_{1\leqslant i<j\leqslant n}\mathsf a_{ij}
\bigl(F(x)-F((ij)\cdot x)\bigr)+
\sum_{g\in G^n}\mathsf c_g
\bigl(F(x)-F(g^{-1}x)\bigr).
\end{equation}
Let $\psi(\mathcal L_{\mathrm{col}})$ be its least eigenvalue on the orthogonal complement of the constant functions.

Combining the preceding identification with Theorem~\ref{thm:universal-reduction}, we obtain the following.

\begin{corollary}
Let $w$ be as in \eqref{eq:intro-weight}.  Then
\[
\psi_{W_n(G)}(w)
=
\min\set{\lambda_{\mathrm{rw}}(\mathsf a),\,\psi(\mathcal L_{\mathrm{col}})}.
\]
\end{corollary}

\section{Decomposition of the quasi-regular representation}\label{sec:constituents}

We now decompose $\mathcal P_n(G)$.  We use the standard parametrisation of wreath-product representations by multipartitions; see~\cite[Section~2.4.1]{CeccheriniScarabottiTolli} or~\cite[Chapter~4]{JamesKerber}.

Recall that $\Gamma=\Irr(G)$.  Choose a representative $U_\gamma$ of each $\gamma\in\Gamma$, and write $d_\gamma=\dim U_\gamma$.  A $G$-multipartition of $n$ is a family $\Lambda=(\lambda^\gamma)_{\gamma\in\Gamma}$ of partitions satisfying $\sum_\gamma\abs{\lambda^\gamma}=n$.  The corresponding irreducible representation of $G\wr\Sn_n$ is
\begin{equation}\label{eq:wreath-irrep-construction}
V_\Lambda
=
\Ind_{\prod_{\gamma\in\Gamma}(G\wr\Sn_{\abs{\lambda^\gamma}})}^{G\wr\Sn_n}
\left(
\mathop{\boxtimes}_{\gamma\in\Gamma}
\bigl(U_\gamma^{\otimes\abs{\lambda^\gamma}}\otimes S^{\lambda^\gamma}\bigr)
\right),
\end{equation}
where the base group acts trivially on each Specht module $S^{\lambda^\gamma}$.  These representations form a complete set of pairwise non-isomorphic irreducible representations of $G\wr\Sn_n$.

For a partition $\lambda$, we write $\Schur_\lambda(\C^d)$ for the irreducible polynomial $\mathrm{GL}_d(\C)$-module of highest weight $\lambda$, and set it equal to zero when $\len(\lambda)>d$; see~\cite[Section~6.1]{FultonHarris}.

\begin{theorem}\label{thm:coset-decomposition}
Let $\Lambda=(\lambda^\gamma)_{\gamma\in\Gamma}$ be a $G$-multipartition of $n$.  Then the multiplicity of $V_\Lambda$ in $\mathcal P_n(G)$ is $\prod_{\gamma\in\Gamma}\dim\Schur_{\lambda^\gamma}(\C^{d_\gamma})$.  Consequently,
\[
\mathcal P_n(G)
\cong
\bigoplus_{\substack{\Lambda=(\lambda^\gamma)_{\gamma\in\Gamma}\
\sum_\gamma\abs{\lambda^\gamma}=n\,,\ 
\len(\lambda^\gamma)\leqslant d_\gamma\text{ for every }\gamma}}
\left(
\prod_{\gamma\in\Gamma}\dim\Schur_{\lambda^\gamma}(\C^{d_\gamma})
\right)V_\Lambda .
\]
\end{theorem}

\begin{proof}
The regular bimodule decomposition is $\C[G]\cong\bigoplus_{\gamma\in\Gamma}U_\gamma\otimes U_\gamma^*$; see~\cite[Section~6.2]{Serre}.  By Section~\ref{subsec:wreath-quasiregular}, $\mathcal P_n(G)$ is $\C[G]^{\otimes n}$, with $G^n$ acting on the tensor factors and $\Sn_n$ permuting them.  Grouping the tensor factors by isotypic type gives
\begin{equation}\label{eq:type-decomposition}
\mathcal P_n(G)
\cong
\bigoplus_{\substack{(m_\gamma)\in\N_0^\Gamma\,;\ 
\sum_\gamma m_\gamma=n}}
\Ind_{\prod_\gamma(G\wr\Sn_{m_\gamma})}^{G\wr\Sn_n}
\left(
\mathop{\boxtimes}_{\gamma\in\Gamma}
\bigl(U_\gamma^{\otimes m_\gamma}\otimes(U_\gamma^*)^{\otimes m_\gamma}\bigr)
\right).
\end{equation}
Schur--Weyl duality~\cite[Section~6.1]{FultonHarris} gives
\begin{equation}\label{eq:schur-weyl}
(U_\gamma^*)^{\otimes m_\gamma}
\cong
\bigoplus_{\substack{\lambda\vdash m_\gamma\,;\ 
\len(\lambda)\leqslant d_\gamma}}
S^\lambda\otimes\Schur_\lambda(U_\gamma^*).
\end{equation}
Substituting \eqref{eq:schur-weyl} into \eqref{eq:type-decomposition} and comparing with \eqref{eq:wreath-irrep-construction} proves the result.
\end{proof}

Combining Theorem~\ref{thm:coset-decomposition} with Lemma~\ref{lem:frobenius-spherical}, we obtain the spherical classification.

\begin{corollary}\label{cor:spherical-classification}
Let $\Lambda=(\lambda^\gamma)_{\gamma\in\Gamma}$ be a $G$-multipartition of $n$.  Then $\dim V_\Lambda^{\Sn_n}=\prod_{\gamma\in\Gamma}\dim\Schur_{\lambda^\gamma}(\C^{d_\gamma})$.  Moreover, $V_\Lambda$ is $\Sn_n$-spherical if and only if $\len(\lambda^\gamma)\leqslant d_\gamma$ for every $\gamma\in\Gamma$.
\end{corollary}

Recall that a pair $(H,K)$ of finite groups with $K\leqslant H$ is a \textbf{Gelfand pair} if $\Ind_K^H\triv$ is multiplicity-free; see~\cite[Section~1.2.4]{CeccheriniScarabottiTolli}.

\begin{corollary}
The quasi-regular representation $\mathcal P_n(G)$ is multiplicity-free if and only if $G$ is abelian.  Equivalently, $(G\wr\Sn_n,\Sn_n)$ is a Gelfand pair if and only if $G$ is abelian.
\end{corollary}

\begin{proof}
If $G$ is abelian, then $d_\gamma=1$ for every $\gamma\in\Gamma$, and every non-zero Schur module in Theorem~\ref{thm:coset-decomposition} is one-dimensional.  Conversely, suppose that $G$ is non-abelian.  Then some $\gamma\in\Gamma$ has $d_\gamma\geqslant2$.  Taking $\lambda^\gamma=(n)$ and all other components empty gives multiplicity $\dim\Symm^n(\C^{d_\gamma})>1$, so $\mathcal P_n(G)$ is not multiplicity-free.
\end{proof}

\subsection{Finite abelian base groups}

Let $A$ be a finite abelian group of order $r$, and let $\widehat A$ be its character group.  A multiplicity vector is a family $\kappa=(\kappa_\chi)_{\chi\in\widehat A}$ of non-negative integers satisfying $\sum_{\chi\in\widehat A}\kappa_\chi=n$.  For such $\kappa$, write $V_\kappa$ for the irreducible wreath-product representation whose $\chi$-component is the one-row partition $(\kappa_\chi)$.

\begin{corollary}\label{cor:cyclic-row-multipartitions}
Let $\kappa$ be a multiplicity vector.  Then $V_\kappa$ is $\Sn_n$-spherical, occurs with multiplicity one in $\mathcal P_n(A)$, and has dimension $n!/\prod_{\chi\in\widehat A}\kappa_\chi!$.  Every spherical irreducible representation arises in this way.  The trivial representation corresponds to $\kappa_\triv=n$ and $\kappa_\chi=0$ for $\chi\neq\triv$.
\end{corollary}

\begin{proof}
A Schur module $\Schur_\lambda(\C)$ is non-zero if and only if $\lambda$ has at most one row, in which case it is one-dimensional.  The dimension formula follows either from \eqref{eq:wreath-irrep-construction} or from the orbit realisation below.
\end{proof}

Following~\cite[Section~1]{FilmusODonnellWu} and~\cite[Section~1.1]{Salez}, the multislice of shape $\kappa$ is
\[
\Omega_\kappa
=
\set{z=(\chi_1,\ldots,\chi_n)\in\widehat A^n:
\abs{\set{i:\chi_i=\chi}}=\kappa_\chi\text{ for every }\chi\in\widehat A}.
\]
The group $\Sn_n$ acts transitively on $\Omega_\kappa$ by permuting coordinates.  For $z=(\chi_1,\ldots,\chi_n)\in\widehat A^n$ and $x=(x_1,\ldots,x_n)\in A^n$, write $\ip{z}{x}=\prod_{i=1}^n\chi_i(x_i)$. 

\begin{proposition}\label{prop:cyclic-orbit-realisation}
Let $\kappa$ be a multiplicity vector.  Then $V_\kappa$ has a basis $\{\mathbf e_z\}_{z\in\Omega_\kappa}$ satisfying
\begin{equation}\label{eq:cyclic-orbit-action}
x\cdot \mathbf e_z
=
\ip{z}{x}\mathbf e_z,
\qquad
\pi\cdot \mathbf e_z=\mathbf e_{\pi\cdot z},
\end{equation}
for $z=(\chi_1,\ldots,\chi_n)\in\Omega_\kappa$, $x=(x_1,\ldots,x_n)\in A^n$, and $\pi\in\Sn_n$.  Moreover,
\begin{equation}\label{eq:cyclic-fourier-decomposition}
\mathcal P_n(A)
\cong
\bigoplus_{\substack{\kappa\in\N_0^{\widehat A}\\
\sum_{\chi\in\widehat A}\kappa_\chi=n}}V_\kappa.
\end{equation}
\end{proposition}

\begin{proof}
Identify $\mathcal P_n(A)$ with $\C[A^n]$ as in Section~\ref{subsec:wreath-quasiregular}.  We use the standard Fourier basis of $\C[A^n]$; see~\cite[Chapter~2]{CSTDiscreteHarmonic}.  For $z=(\chi_1,\ldots,\chi_n)\in\widehat A^n$, put
\[
\mathbf e_z
=
\frac{1}{\abs{A}^{n/2}}
\sum_{x\in A^n}
\overline{\ip{z}{x}}\,\delta_{x}.
\]
A direct calculation gives \eqref{eq:cyclic-orbit-action}.  Hence the Fourier basis decomposes into the orbit spaces $\C[\Omega_\kappa]$, giving \eqref{eq:cyclic-fourier-decomposition}.

Each orbit space is irreducible.  Indeed, its restriction to $A^n$ is a direct sum of pairwise distinct one-dimensional characters, one for each $z\in\Omega_\kappa$.  Every $A^n$-invariant subspace is therefore spanned by a subset of the Fourier basis.  Invariance under $\Sn_n$, which acts transitively on $\Omega_\kappa$, forces that subset to be empty or the whole orbit.  Distinct multiplicity vectors give disjoint $\Sn_n$-orbits of characters, so the orbit spaces are pairwise non-isomorphic.  Corollary~\ref{cor:cyclic-row-multipartitions} identifies them with the representations $V_\kappa$.
\end{proof}

\begin{corollary}
The family consisting of $\Std_n$ and all non-trivial $V_\kappa$ has cardinality $\binom{n+r-1}{r-1}$, and the sum of the dimensions of its members is $r^n+n-2$.
\end{corollary}

\begin{proof}
There are $\binom{n+r-1}{r-1}$ multiplicity vectors on the $r$ characters of $A$.  We remove the vector indexing the trivial representation and add $\Std_n$, so the number is unchanged.  The multinomial identity $\sum_{\sum_\chi \kappa_\chi=n}n!/\prod_\chi \kappa_\chi!=r^n$ gives the sum of the dimensions.  Removing the trivial one-dimensional summand and adding the $(n-1)$-dimensional standard representation gives $r^n+n-2$.
\end{proof}

\section{Schr\"odinger operators on multislices}\label{sec:cyclic-blocks}

We now specialise to abelian base groups and realise the Laplacian on each spherical irreducible representation as a discrete Schr\"odinger operator on a multislice.

Let $A$ be a finite abelian group, and let $w_N=\sum_{x\in A^n}\mathsf c_x\,x\in\R[A^n]$ be symmetric and non-negative.  In the Fourier basis, the base-group Laplacian is multiplication by
\begin{equation}\label{eq:fourier-potential}
Q(z)
=
\sum_{x\in A^n}\mathsf c_x\bigl(1-\operatorname{Re}\ip{z}{x}\bigr);
\end{equation}
see~\cite[Chapter~2]{CSTDiscreteHarmonic}.  Thus $Q$ is real and non-negative, and $Q(z^{-1})=Q(z)$.

Let $w_T=\sum_{i<j}\mathsf a_{ij}(ij)$.  For a multiplicity vector $\kappa$, let $\mathcal L_\kappa$ be the operator on $\C[\Omega_\kappa]$ given by
\begin{equation}\label{eq:schrodinger-block}
(\mathcal L_\kappa f)(z)
=
\sum_{1\leqslant i<j\leqslant n}\mathsf a_{ij}
\bigl(f(z)-f((ij)\cdot z)\bigr)
+
Q(z)f(z).
\end{equation}

\begin{proposition}\label{prop:block-identification}
Let $\kappa$ be a multiplicity vector.  Then $V_\kappa(\Delta_{A\wr\Sn_n}(w_T+w_N))=\mathcal L_\kappa$ in the realisation of Proposition~\ref{prop:cyclic-orbit-realisation}.  If $V_\kappa$ is non-trivial, then $\psi_{A\wr\Sn_n}(w_T+w_N,V_\kappa)=\lambda_{\min}(\mathcal L_\kappa)$.
\end{proposition}

\begin{proof}
Equation~\eqref{eq:cyclic-orbit-action} gives the actions of $(ij)$ and $A^n$ on the basis $\{\mathbf e_z\}_{z\in\Omega_\kappa}$.  Substitution in the definition of the Laplacian gives \eqref{eq:schrodinger-block}.
\end{proof}

Combining Theorem~\ref{thm:universal-reduction}, Corollary~\ref{cor:cyclic-row-multipartitions}, and Proposition~\ref{prop:block-identification}, we obtain the following result.

\begin{theorem}\label{thm:cyclic-characterisation}
Let $A$ be a finite abelian group, let $n\geqslant2$, and let $w$ be as in \eqref{eq:intro-weight}.  Then
\begin{equation}\label{eq:explicit-cyclic-characterisation}
\psi_{A\wr\Sn_n}(w)
=
\min\left\{
\lambda_{\mathrm{rw}}(\mathsf a),
\min_{\substack{\sum_{\chi\in\widehat A}\kappa_\chi=n\,;\ 
V_\kappa\neq\triv}}
\lambda_{\min}(\mathcal L_\kappa)
\right\}.
\end{equation}
\end{theorem}

For $A=C_2$, the multiplicity vectors are $(n-k,k)$ and $V_{(n-k,k)}$ is the bipartition representation $((n-k),(k))$.  Thus \eqref{eq:explicit-cyclic-characterisation} recovers the characterisation of Alon and Ghosh~\cite[Theorem~1]{AlonGhosh}.

The single-coordinate transitions considered by Ghosh~\cite[Definition~1.1]{Ghosh} have
\[
w_N=\sum_{i=1}^n\mathsf y_i\sum_{g\in A}\mathsf c_g\,g^{(i)},
\qquad
\mathsf c_g=\mathsf c_{g^{-1}}\geqslant0,
\]
where $g^{(i)}$ has entry $g$ in coordinate $i$ and the identity elsewhere.  In this case,
\[
Q(z)
=
\sum_{i=1}^n\mathsf y_i\sum_{g\in A}\mathsf c_g
\bigl(1-\operatorname{Re}\chi_i(g)\bigr).
\]
Allowing arbitrary elements of $A^n$ gives a general non-negative potential on $\widehat A^n$.

\section{Minimal gap-determining families}\label{sec:minimality}

We prove that the reduction in Theorem~\ref{thm:universal-reduction} is sharp.  We first treat abelian base groups by Fourier inversion and then pass to arbitrary finite groups through the spherical Hecke algebra.

Let $G$ be a finite group.  A collection $\mathcal E\subseteq\Irr(W_n(G))\setminus\{\triv\}$ is \textbf{gap-determining} if $\psi_{W_n(G)}(w)=\min_{V\in\mathcal E}\psi_{W_n(G)}(w,V)$ for every $w$ of the form \eqref{eq:intro-weight}.  It is \textbf{minimal} if no proper subcollection has this property.

We begin with finite abelian groups.  For a multiplicity vector $\kappa$, let $(\kappa^*)_\chi=\kappa_{\chi^{-1}}$.  Then $\Omega_\kappa^{-1}=\Omega_{\kappa^*}$.

\begin{proposition}\label{prop:duality-isospectral}
Let $w$ be as in \eqref{eq:intro-weight}, and let $\kappa$ be a multiplicity vector.  Then $\mathcal L_\kappa$ and $\mathcal L_{\kappa^*}$ are isospectral, and $V_{\kappa^*}\cong V_\kappa^*$.  In particular, $V_\kappa$ and $V_{\kappa^*}$ have the same Laplacian spectrum for every real symmetric group-algebra element.
\end{proposition}

\begin{proof}
The unitary map $J:\C[\Omega_\kappa]\to\C[\Omega_{\kappa^*}]$ given by $(Jf)(z)=f(z^{-1})$ intertwines $\mathcal L_\kappa$ and $\mathcal L_{\kappa^*}$.  The assertion about contragredients follows from \eqref{eq:cyclic-orbit-action}.
\end{proof}

Let $T_n=\sum_{1\leqslant i<j\leqslant n}(ij)$.  The content formula for the transposition class sum~\cite[Example~3.15]{Ryba} gives the following spectral separation.

\begin{lemma}\label{lem:complete-transposition-gap}
Let $U$ be a non-trivial irreducible representation of $\Sn_n$.  Then $\Delta_{\Sn_n}(T_n)$ acts as the scalar $n$ on $\Std_n$; if $U\not\cong\Std_n$, every eigenvalue of $U(\Delta_{\Sn_n}(T_n))$ is strictly greater than $n$.
\end{lemma}

\subsection{Separating a prescribed one-row multipartition}

Fix a multiplicity vector $\kappa$ for which $V_\kappa$ is non-trivial.  For $x\in A^n$, let $\Phi_\kappa(x)=\abs{\Omega_\kappa}^{-1}\sum_{z\in\Omega_\kappa}\operatorname{Re}\ip{z}{x}$.  Given $C>\eps>0$, take $\mathsf c_x=C+\eps\Phi_\kappa(x)$.

\begin{lemma}\label{lem:separating-fourier-potential}
Let $C>\eps>0$.  Then the coefficients $\mathsf c_x$ are strictly positive, invariant under coordinate permutations, and satisfy $\mathsf c_x=\mathsf c_{x^{-1}}$.  Moreover, for every non-trivial $y\in\widehat A^n$,
\begin{equation}\label{eq:separating-potential}
Q(y)
=
\begin{cases}
C\abs{A}^n-\dfrac{\eps\abs{A}^n}{2\abs{\Omega_\kappa}},
& \kappa\neq\kappa^*,\ y\in\Omega_\kappa\cup\Omega_{\kappa^*},\\[2mm]
C\abs{A}^n-\dfrac{\eps\abs{A}^n}{\abs{\Omega_\kappa}},
& \kappa=\kappa^*,\ y\in\Omega_\kappa,\\[2mm]
C\abs{A}^n,
& \text{otherwise}.
\end{cases}
\end{equation}
\end{lemma}

\begin{proof}
The first assertion follows from $\abs{\Phi_\kappa(x)}\leqslant1$ and the symmetries of $\Omega_\kappa$.  Character orthogonality~\cite[Section~2.3]{Serre} shows that the Fourier transform of $\Phi_\kappa$ is supported on $\Omega_\kappa\cup\Omega_{\kappa^*}$.  Substitution in \eqref{eq:fourier-potential} gives \eqref{eq:separating-potential}; the two cases reflect whether these orbits are distinct.
\end{proof}

\begin{theorem}\label{thm:row-separation}
Let $A$ be a finite abelian group, and let $\kappa$ be a multiplicity vector for which $V_\kappa$ is non-trivial.  Then there is an element $w$ of the form \eqref{eq:intro-weight} for which $\psi_{A\wr\Sn_n}(w,U)=\psi_{A\wr\Sn_n}(w)$ holds precisely when $U\cong V_\kappa$ or $U\cong V_{\kappa^*}$.
\end{theorem}

\begin{proof}
Choose $C>\eps>0$ and the corresponding weights $\mathsf c_x=C+\eps\Phi_\kappa(x)$.  Let $w_N=\sum_{x\in A^n}\mathsf c_x\,x$, $w_T=C\abs{A}^nT_n/n$, and $w=w_T+w_N$.  The base-group weights are strictly positive and symmetric.

Equation~\eqref{eq:separating-potential} shows that $Q$ is strictly less than $C\abs{A}^n$ on $\Omega_\kappa\cup\Omega_{\kappa^*}$ and equals $C\abs{A}^n$ on every other non-trivial multislice.  On each $V_\mu$, the transposition Laplacian is positive semidefinite and annihilates the constant vector.  Thus precisely $V_\kappa$ and $V_{\kappa^*}$ have least eigenvalue below $C\abs{A}^n$ among the spherical irreducible representations.

Lemma~\ref{lem:complete-transposition-gap} gives $\psi(w,\Std_n)=C\abs{A}^n$, and Proposition~\ref{prop:domination} gives the same lower bound outside the spherical family.  The result follows.
\end{proof}

\subsection{Separating the lifted standard representation}

We next show that the lifted standard representation can be separated from every other irreducible representation.

\begin{theorem}\label{thm:standard-separation}
Let $G$ be a finite group.  Then there is an element $w$ of the form \eqref{eq:intro-weight} for which $\psi_{W_n(G)}(w,U)=\psi_{W_n(G)}(w)$ holds precisely when $U\cong\Std_n$.
\end{theorem}

\begin{proof}
Choose $C>0$ with $C\abs{G}^n>n$, and set $w_T=T_n$ and $w_N=C\sum_{x\in G^n}x$.
In an irreducible representation $V$ of $W_n(G)$, the operator $\abs{G}^{-n}\sum_{x\in G^n}V(x)$ is the orthogonal projection onto $V^{G^n}$~\cite[Section~4.3]{Serre}.  Since $G^n$ is normal, this fixed subspace is invariant under the whole wreath product.  Hence it is either zero or all of $V$.

If $G^n$ acts non-trivially on $V$, then $V^{G^n}=0$ and the base-group Laplacian is the scalar $C\abs{G}^n$.  Therefore $\psi(w,V)\geqslant C\abs{G}^n>n$.

If $G^n$ acts trivially on $V$, then $V$ is lifted from an irreducible representation of $\Sn_n$, and the base-group Laplacian vanishes.  Lemma~\ref{lem:complete-transposition-gap} shows that the least non-trivial value is $n$, with equality only for $\Std_n$.
\end{proof}

Combining Theorems~\ref{thm:cyclic-characterisation}, \ref{thm:row-separation}, and~\ref{thm:standard-separation} with Proposition~\ref{prop:duality-isospectral}, we obtain the abelian classification.

\begin{theorem}
Let $\mathcal M_n(A)$ consist of $\Std_n$ and one representation $V_\kappa$ from each duality orbit of non-trivial multiplicity vectors.  Then $\mathcal M_n(A)$ is gap-determining and minimal.  Moreover, every gap-determining family contains $\Std_n$ and at least one of $V_\kappa$ and $V_{\kappa^*}$ for each non-trivial multiplicity vector $\kappa$.
\end{theorem}

The cardinality of a minimal gap-determining family can also be read from the inversion action on $\widehat A$.  Write $\widehat A[2]=\{\chi\in\widehat A:\chi^2=\triv\}$ for the subgroup of characters of order at most two.

\begin{corollary}
Let $A$ be a finite abelian group of order $r$.  Every minimal gap-determining family for $A\wr\Sn_n$ has cardinality
\[
\frac12\left(
\binom{n+r-1}{r-1}
+
[t^n](1-t)^{-\abs{\widehat A[2]}}
(1-t^2)^{-(r-\abs{\widehat A[2]})/2}
\right).
\]
\end{corollary}

\begin{proof}
Inversion fixes the $\abs{\widehat A[2]}$ characters in $\widehat A[2]$ and partitions the remaining characters into $(r-\abs{\widehat A[2]})/2$ pairs.  A multiplicity vector fixed by inversion assigns arbitrary multiplicities to the fixed characters and equal multiplicities to the two characters in each pair.  Its generating function is therefore $(1-t)^{-\abs{\widehat A[2]}}(1-t^2)^{-(r-\abs{\widehat A[2]})/2}$.
Burnside's orbit-counting lemma~\cite[Section~2.3]{CameronPermutationGroups} gives the displayed number of orbits of the set of all multiplicity vectors.  Removing the orbit corresponding to the trivial representation and adding the lifted standard representation leaves the cardinality unchanged.
\end{proof}

\subsection{General finite groups}

We now extend the separation argument to arbitrary finite groups, replacing Fourier inversion by the spherical Hecke algebra.

For a complex representation $V$ of a finite group, let $V^*$ denote its contragredient representation.  For $a=\sum_x\mathsf c_x\,x$ in a complex group algebra, write $a^*=\sum_x\overline{\mathsf c_x}\,x^{-1}$ and $\overline a=\sum_x\overline{\mathsf c_x}\,x$.

The following observation explains why representations can be separated only up to duality.

\begin{proposition}\label{prop:complex-conjugate-spectra}
Let $H$ be a finite group, let $w\in\R[H]$ be symmetric and non-negative, and let $V$ be a complex representation of $H$.  Then $V(\Delta_H(w))$ and $V^*(\Delta_H(w))$ are isospectral.
\end{proposition}

\begin{proof}
Choose a unitary matrix realisation $\rho$ of $V$; then $\overline\rho$ realises $V^*$.  Since the coefficients of $\Delta_H(w)$ are real, $\overline\rho(\Delta_H(w))=\overline{\rho(\Delta_H(w))}$.
Both matrices are Hermitian, so their eigenvalues are real.  Complex conjugation therefore preserves the spectrum.
\end{proof}

Thus a gap-determining family need not contain both members of a duality orbit.

Let $e_n=n!^{-1}\sum_{\sigma\in\Sn_n}\sigma$.  The \textbf{spherical Hecke algebra} of $(W_n(G),\Sn_n)$ is $\mathcal H_n(G)=e_n\C[W_n(G)]e_n$.  Equivalently, it is the convolution algebra of functions constant on the double cosets of $\Sn_n$; see~\cite[Sections~1.2.2--1.2.4]{CeccheriniScarabottiTolli}.

\begin{lemma}
Let $G$ be a finite group.  Then the Wedderburn decomposition~\cite[Section~6.2]{Serre} restricts to an algebra isomorphism
\begin{equation}\label{eq:hecke-decomposition}
\mathcal H_n(G)
\cong
\bigoplus_{U\in\Sph_n(G)}\End(U^{\Sn_n}).
\end{equation}
Moreover, $e_n\C[W_n(G)]e_n=e_n\C[G^n]e_n$.  Under \eqref{eq:hecke-decomposition}, the $U$-block of $h\in\mathcal H_n(G)$ is the restriction of $U(h)$ to $U^{\Sn_n}$.
\end{lemma}

\begin{proof}
The Wedderburn decomposition~\cite[Section~6.2]{Serre} gives $\C[W_n(G)]\cong\bigoplus_{U\in\Irr(W_n(G))}\End(U)$.
In the $U$-summand, the element $e_n$ is the orthogonal projection from $U$ onto $U^{\Sn_n}$.  Hence
\[
e_n\C[W_n(G)]e_n
\cong
\bigoplus_{U\in\Irr(W_n(G))}
U(e_n)\End(U)U(e_n)
\cong
\bigoplus_{U\in\Sph_n(G)}\End(U^{\Sn_n}),
\]
which proves \eqref{eq:hecke-decomposition}.

Every element of $W_n(G)=G^n\Sn_n$ has the form $x\sigma$ with $x\in G^n$ and $\sigma\in\Sn_n$.  Since $\sigma e_n=e_n$, we have $e_nx\sigma e_n=e_nx e_n$.  The elements $x\sigma$ span $\C[W_n(G)]$, so the asserted equality follows.
\end{proof}

For $b\in\R[G^n]$, denote the coefficient of $x$ by $b_x$.  We also require a real symmetric lift from the spherical Hecke algebra.

\begin{lemma}\label{lem:real-hecke-lift}
Let $p\in\mathcal H_n(G)$ be self-adjoint and fixed by coefficient conjugation.  Then there is an element $b\in\R[G^n]$ which commutes with $\Sn_n$ and satisfies $e_nbe_n=p$ and $b_x=b_{x^{-1}}$ for every $x\in G^n$.
\end{lemma}

\begin{proof}
By the preceding lemma, choose $b_0\in\C[G^n]$ with $e_nb_0e_n=p$, and average it under conjugation by $\Sn_n$, setting $b_1=n!^{-1}\sum_{\sigma\in\Sn_n}\sigma b_0\sigma^{-1}$.  Since $e_n$ is fixed by coefficient conjugation and by the involution $a\mapsto a^*$, the element $b=\frac14(b_1+b_1^*+\overline{b_1}+\overline{b_1^*})$
still satisfies $e_nbe_n=p$.  Its coefficients are real and invariant under inversion, and the first averaging ensures that it commutes with $\Sn_n$.
\end{proof}

The preceding two lemmas now give the required separation.

\begin{theorem}\label{thm:general-spherical-separation}
Let $V\in\Sph_n(G)\setminus\{\triv\}$.  Then there is an element $w$ of the form \eqref{eq:intro-weight} for which $\psi_{W_n(G)}(w,U)=\psi_{W_n(G)}(w)$ holds precisely when $U\cong V$ or $U\cong V^*$.
\end{theorem}

\begin{proof}
Let $\mathcal O$ be the set of distinct isomorphism classes among $V$ and $V^*$.  Under \eqref{eq:hecke-decomposition}, let $p_{\mathcal O}$ be the central idempotent that acts as the identity on $U^{\Sn_n}$ for $U\in\mathcal O$ and as zero on $U^{\Sn_n}$ for every other spherical irreducible representation $U$.  The element $p_{\mathcal O}$ is self-adjoint.  It is also fixed by coefficient conjugation, because coefficient conjugation interchanges the blocks indexed by $U$ and $U^*$.

By Lemma~\ref{lem:real-hecke-lift}, choose $b\in\R[G^n]$ which commutes with $\Sn_n$ and satisfies $e_nbe_n=p_{\mathcal O}$.  Write $b=\sum_{x\in G^n}b_x\,x$; then $b_x=b_{x^{-1}}$.  Since $p_{\mathcal O}$ vanishes in the trivial spherical block, $\sum_{x\in G^n}b_x=0$.

Choose $C>0$ and then $\eta>0$ so small that $C+\eta b_x>0$ for every $x\in G^n$ and
$\eta\bigl(1+\sum_x\abs{b_x}\bigr)<C\abs{G}^n/2$.  Let $w_N=\sum_x(C+\eta b_x)\,x$, $w_T=C\abs{G}^nT_n/n$, and $w=w_T+w_N$.  Thus $w$ has the form \eqref{eq:intro-weight}, with strictly positive base-group weights.

Let $U$ be an irreducible representation of $W_n(G)$.  Since $G^n$ is normal, $U^{G^n}$ is either zero or all of $U$.  The averaging operator over $G^n$~\cite[Section~4.3]{Serre}, together with $\sum_x b_x=0$, gives
\begin{equation}\label{eq:general-diagonal-laplacian}
U(\Delta(w_N))
=
\begin{cases}
0,&U^{G^n}=U,\\[1mm]
C\abs{G}^nI-\eta U(b),&U^{G^n}=0.
\end{cases}
\end{equation}
Moreover, $U(b)$ is self-adjoint and $\norm{U(b)}\leqslant\sum_{x\in G^n}\abs{b_x}$.  Because $b$ commutes with $\Sn_n$, it preserves every $\Sn_n$-isotypic component of $U$.  On the $\Sn_n$-fixed subspace, the identity $e_nbe_n=p_{\mathcal O}$ gives
\begin{equation}\label{eq:b-on-fixed-space}
U(b)\big|_{U^{\Sn_n}}
=
\begin{cases}
I,&U\in\mathcal O,\\
0,&U\in\Sph_n(G)\setminus\mathcal O.
\end{cases}
\end{equation}

Consider first a non-trivial $\Sn_n$-isotypic component of $U$.  By Lemma~\ref{lem:complete-transposition-gap}, the transposition Laplacian is at least $C\abs{G}^nI$ on this component.  If $U^{G^n}=0$, then \eqref{eq:general-diagonal-laplacian} shows that the full Laplacian is greater than $C\abs{G}^n$, because $2C\abs{G}^n-\eta\sum_{x\in G^n}\abs{b_x}>C\abs{G}^n$.
If $U^{G^n}=U$, then $U$ is lifted from $\Sn_n$; its gap is $C\abs{G}^n$ when $U=\Std_n$ and is strictly greater than $C\abs{G}^n$ for every other non-trivial lifted irreducible representation.

A non-trivial spherical irreducible representation cannot be trivial on $G^n$, because a lifted irreducible representation has an $\Sn_n$-fixed vector only when it is trivial.  Hence \eqref{eq:general-diagonal-laplacian} and \eqref{eq:b-on-fixed-space} show that the least eigenvalue on $U^{\Sn_n}$ is $C\abs{G}^n-\eta$ for $U\in\mathcal O$, and is $C\abs{G}^n$ for every other non-trivial spherical $U$.  All eigenvalues on non-trivial $\Sn_n$-isotypic components are strictly greater than $C\abs{G}^n$.  Thus equality with the spectral gap holds precisely for the members of $\mathcal O$.
\end{proof}

In Theorem~\ref{thm:general-spherical-separation}, every base-group weight may be chosen strictly positive and all transpositions may be assigned the same positive rate.  If $V\cong V^*$, no other non-trivial irreducible representation has the same gap.

Combining Theorems~\ref{thm:universal-reduction}, \ref{thm:standard-separation}, and~\ref{thm:general-spherical-separation} with Proposition~\ref{prop:complex-conjugate-spectra}, we obtain the following classification.

\begin{theorem}\label{thm:general-minimal-family}
Let $\mathcal M_n(G)$ consist of $\Std_n$ and one representation from each duality orbit in $\Sph_n(G)\setminus\{\triv\}$.  Then $\mathcal M_n(G)$ is gap-determining and minimal.  Moreover, every gap-determining family contains $\Std_n$ and meets every duality orbit in $\Sph_n(G)\setminus\{\triv\}$.
\end{theorem}

Duality is transparent in the multipartition parametrisation.  The dual multipartition $\Lambda^*$ is given by $(\Lambda^*)^\gamma=\lambda^{\gamma^*}$, where $\gamma^*$ denotes the contragredient of $\gamma\in\Gamma$.  Since Specht modules may be realised over $\R$~\cite[Section~2.1]{JamesKerber}, \eqref{eq:wreath-irrep-construction} gives $V_\Lambda^*\cong V_{\Lambda^*}$.

\begin{samepage}
\begin{corollary}
Let $a_n(G)$ be the number of $\Sn_n$-spherical irreducible representations of $W_n(G)$, and let $f_n(G)$ be the number of these representations self-dual.  Then every minimal gap-determining family for $W_n(G)$ has cardinality $(a_n(G)+f_n(G))/2$.  Moreover,
\begin{align}
a_n(G)
&=[t^n]
\prod_{\gamma\in\Gamma}
\prod_{j=1}^{d_\gamma}\frac{1}{1-t^j},
\label{eq:general-spherical-count}\\
f_n(G)
&=[t^n]
\prod_{\substack{\gamma\in\Gamma\\\gamma\cong\gamma^*}}
\prod_{j=1}^{d_\gamma}\frac{1}{1-t^j}
\prod_{\{\gamma,\gamma^*\}:\,\gamma\not\cong\gamma^*}
\prod_{j=1}^{d_\gamma}\frac{1}{1-t^{2j}}.
\label{eq:general-fixed-count}
\end{align}
In the second product of \eqref{eq:general-fixed-count}, each unordered dual pair is taken once.
\end{corollary}
\end{samepage}

\begin{proof}
By Corollary~\ref{cor:spherical-classification}, the component $\lambda^\gamma$ may be any partition with at most $d_\gamma$ parts.  Multiplying the corresponding partition-generating functions gives \eqref{eq:general-spherical-count}.  A multipartition is self-dual precisely when $\lambda^\gamma=\lambda^{\gamma^*}$ for every $\gamma$.  A self-dual irreducible representation of $G$ contributes the same factor as before, while a non-self-dual pair contributes only even total sizes, giving \eqref{eq:general-fixed-count}.

Burnside's orbit-counting lemma~\cite[Section~2.3]{CameronPermutationGroups} shows that duality has $(a_n(G)+f_n(G))/2$ orbits on the set of all spherical irreducible representations.  The trivial orbit is removed from the spherical family and $\Std_n$ is added, so the cardinality remains unchanged.
\end{proof}

\end{document}